\documentclass[11pt]{amsart}
\usepackage{amsmath,amsfonts,mathrsfs}

\newif\ifdraft
\draftfalse
\usepackage{amsmath,amsfonts,amsthm,mathrsfs}
\usepackage{amssymb}

\usepackage[unicode,bookmarks]{hyperref}
\usepackage[usenames,dvipsnames]{xcolor}
\hypersetup{colorlinks=true,citecolor=NavyBlue,linkcolor=Brown,urlcolor=Orange}

\usepackage[alphabetic,initials]{amsrefs}

\usepackage{enumitem}

\usepackage{chngcntr}

\ifdraft
\usepackage[notcite,notref,color]{showkeys}

\definecolor{labelkey}{gray}{0.5}
\fi

\usepackage{tikz}

\usepackage{tikz-cd}

\usetikzlibrary{matrix,arrows}
\newlength{\myarrowsize} 

\pgfarrowsdeclare{cmto}{cmto}{
	\pgfsetdash{}{0pt} 
	\pgfsetbeveljoin 
	\pgfsetroundcap 
	\setlength{\myarrowsize}{0.6pt}
	\addtolength{\myarrowsize}{.5\pgflinewidth}
	\pgfarrowsleftextend{-4\myarrowsize-.5\pgflinewidth} 
	\pgfarrowsrightextend{.8\pgflinewidth}
}{
	\setlength{\myarrowsize}{0.6pt} 
  	\addtolength{\myarrowsize}{.5\pgflinewidth}  
	\pgfsetlinewidth{0.5\pgflinewidth}
	\pgfsetroundjoin
	\pgfpathmoveto{\pgfpoint{1.5\pgflinewidth}{0}}
	\pgfpatharc{-109}{-170}{4\myarrowsize}
	\pgfpatharc{10}{189}{0.58\pgflinewidth and 0.2\pgflinewidth}
	\pgfpatharc{-170}{-115}{4\myarrowsize+\pgflinewidth}
	\pgfpathclose
	\pgfusepathqfillstroke
	\pgfpathmoveto{\pgfpoint{1.5\pgflinewidth}{0}}
	\pgfpatharc{109}{170}{4\myarrowsize}
	\pgfpatharc{-10}{-189}{0.58\pgflinewidth and 0.2\pgflinewidth}
	\pgfpatharc{170}{115}{4\myarrowsize+\pgflinewidth}
	\pgfpathclose
	\pgfusepathqfillstroke
	\pgfsetlinewidth{2\pgflinewidth}
}

\pgfarrowsdeclare{cmonto}{cmonto}{
	\pgfsetdash{}{0pt} 
	\pgfsetbeveljoin 
	\pgfsetroundcap 
	\setlength{\myarrowsize}{0.6pt}
	\addtolength{\myarrowsize}{.5\pgflinewidth}
	\pgfarrowsleftextend{-4\myarrowsize-.5\pgflinewidth} 
	\pgfarrowsrightextend{.8\pgflinewidth}
}{
	\setlength{\myarrowsize}{0.6pt} 
  	\addtolength{\myarrowsize}{.5\pgflinewidth}  
	\pgfsetlinewidth{0.5\pgflinewidth}
	\pgfsetroundjoin
	\pgfpathmoveto{\pgfpoint{1.5\pgflinewidth}{0}}
	\pgfpatharc{-109}{-170}{4\myarrowsize}
	\pgfpatharc{10}{189}{0.58\pgflinewidth and 0.2\pgflinewidth}
	\pgfpatharc{-170}{-115}{4\myarrowsize+\pgflinewidth}
	\pgfpathclose
	\pgfusepathqfillstroke
	\pgfpathmoveto{\pgfpoint{1.5\pgflinewidth}{0}}
	\pgfpatharc{109}{170}{4\myarrowsize}
	\pgfpatharc{-10}{-189}{0.58\pgflinewidth and 0.2\pgflinewidth}
	\pgfpatharc{170}{115}{4\myarrowsize+\pgflinewidth}
	\pgfpathclose
	\pgfusepathqfillstroke
	\pgfpathmoveto{\pgfpoint{1.5\pgflinewidth-0.3em}{0}}
	\pgfpatharc{-109}{-170}{4\myarrowsize}
	\pgfpatharc{10}{189}{0.58\pgflinewidth and 0.2\pgflinewidth}
	\pgfpatharc{-170}{-115}{4\myarrowsize+\pgflinewidth}
	\pgfpathclose
	\pgfusepathqfillstroke
	\pgfpathmoveto{\pgfpoint{1.5\pgflinewidth-0.3em}{0}}
	\pgfpatharc{109}{170}{4\myarrowsize}
	\pgfpatharc{-10}{-189}{0.58\pgflinewidth and 0.2\pgflinewidth}
	\pgfpatharc{170}{115}{4\myarrowsize+\pgflinewidth}
	\pgfpathclose
	\pgfusepathqfillstroke
	\pgfsetlinewidth{2\pgflinewidth}
}

\pgfarrowsdeclare{cmhook}{cmhook}{
	\pgfsetdash{}{0pt} 
	\pgfsetbeveljoin 
	\pgfsetroundcap 
	\setlength{\myarrowsize}{0.6pt}
	\addtolength{\myarrowsize}{.5\pgflinewidth}
	\pgfarrowsleftextend{-4\myarrowsize-.5\pgflinewidth} 
	\pgfarrowsrightextend{.8\pgflinewidth}
}{
	\setlength{\myarrowsize}{0.6pt} 
  	\addtolength{\myarrowsize}{.5\pgflinewidth}  
 	\pgfsetdash{}{0pt}
	\pgfsetroundcap
	\pgfpathmoveto{\pgfqpoint{0pt}{-4.667\pgflinewidth}}
	\pgfpathcurveto
    {\pgfqpoint{4\pgflinewidth}{-4.667\pgflinewidth}}
    {\pgfqpoint{4\pgflinewidth}{0pt}}
    {\pgfpointorigin}
	\pgfusepathqstroke
}

\newenvironment{diagram*}[2]{%
\[%
\begin{tikzpicture}[>=cmto,baseline=(current bounding box.center),%
	to/.style={->,font=\scriptsize,cap=round},%
	into/.style={cmhook->,font=\scriptsize,cap=round},%
	onto/.style={-cmonto,font=\scriptsize,cap=round},%
	math/.style={matrix of math nodes, row sep=#2, column sep=#1,%
		text height=1.5ex, text depth=0.25ex}]%
}{%
\end{tikzpicture}%
\]%
\ignorespacesafterend%
}

\newcommand{\Dmod}{\mathscr{D}}
\newcommand{\Mmod}{\mathcal{M}}

\newcommand{\NN}{\mathbb{N}}

\newcommand{\QQ}{\mathbb{Q}}

\newcommand{\CC}{\mathbb{C}}

\newcommand{\shf}[1]{\mathscr{#1}}

\def\overbar#1#2#3{{%
	\setbox0=\hbox{\displaystyle{#1}}%
	\dimen0=\wd0
	\advance\dimen0 by -#2 
	\vbox {\nointerlineskip \moveright #3 \vbox{\hrule height 0.3pt width \dimen0}%
		\nointerlineskip \vskip 1.5pt \box0}%
}}

\newcommand{\shO}{\shf{O}}

\makeatletter
\let\@@seccntformat\@seccntformat
\renewcommand*{\@seccntformat}[1]{%
  \expandafter\ifx\csname @seccntformat@#1\endcsname\relax
    \expandafter\@@seccntformat
  \else
    \expandafter
      \csname @seccntformat@#1\expandafter\endcsname
  \fi
    {#1}%
}
\newcommand*{\@seccntformat@subsection}[1]{%
  \textbf{\csname the#1\endcsname.}
}
\makeatother

\makeatletter
\let\@paragraph\paragraph
\renewcommand*{\paragraph}[1]{%
	\vspace{0.3\baselineskip}%
	\@paragraph{\textit{#1}}%
}
\makeatother

\counterwithin{equation}{section}
\counterwithin{figure}{section}

\newtheorem{theorem}[equation]{Theorem}
\newtheorem*{theorem*}{Theorem}
\newtheorem{lemma}[equation]{Lemma}
\newtheorem*{lemma*}{Lemma}
\newtheorem{corollary}[equation]{Corollary}
\newtheorem{proposition}[equation]{Proposition}
\newtheorem*{proposition*}{Proposition}

\theoremstyle{definition}
\newtheorem{definition}[equation]{Definition}
\newtheorem*{definition*}{Definition}
\newtheorem{remark}[equation]{Remark}

\newtheorem{question}[equation]{Question}

\newtheorem{example}[equation]{Example}
\newtheorem*{example*}{Example}
\newtheorem*{problem*}{Problem}

\theoremstyle{plain}

\newcommand{\theoremref}[1]{\hyperref[#1]{Theorem~\ref*{#1}}}
\newcommand{\lemmaref}[1]{\hyperref[#1]{Lemma~\ref*{#1}}}
\newcommand{\definitionref}[1]{\hyperref[#1]{Definition~\ref*{#1}}}
\newcommand{\propositionref}[1]{\hyperref[#1]{Proposition~\ref*{#1}}}
\newcommand{\conjectureref}[1]{\hyperref[#1]{Conjecture~\ref*{#1}}}
\newcommand{\corollaryref}[1]{\hyperref[#1]{Corollary~\ref*{#1}}}
\newcommand{\exampleref}[1]{\hyperref[#1]{Example~\ref*{#1}}}

\makeatletter
\let\old@caption\caption
\renewcommand*{\caption}[1]{%
	\setcounter{figure}{\value{equation}}%
	\stepcounter{equation}%
	\old@caption{#1}\relax%
}
\makeatother

\newcounter{intro}

\newtheorem{intro-conjecture}[intro]{Conjecture}
\newtheorem{intro-corollary}[intro]{Corollary}
\newtheorem{intro-theorem}[intro]{Theorem}

\newcommand{\parref}[1]{\hyperref[#1]{\S\ref*{#1}}}

\makeatletter
\newcommand*\if@single[3]{%
  \setbox0\hbox{${\mathaccent"0362{#1}}^H$}%
  \setbox2\hbox{${\mathaccent"0362{\kern0pt#1}}^H$}%
  \ifdim\ht0=\ht2 #3\else #2\fi
  }
\newcommand*\rel@kern[1]{\kern#1\dimexpr\macc@kerna}
\newcommand*\widebar[1]{\@ifnextchar^{{\wide@bar{#1}{0}}}{\wide@bar{#1}{1}}}
\newcommand*\wide@bar[2]{\if@single{#1}{\wide@bar@{#1}{#2}{1}}{\wide@bar@{#1}{#2}{2}}}
\newcommand*\wide@bar@[3]{%
  \begingroup
  \def\mathaccent##1##2{%
    \if#32 \let\macc@nucleus\first@char \fi
    \setbox\z@\hbox{$\macc@style{\macc@nucleus}_{}$}%
    \setbox\tw@\hbox{$\macc@style{\macc@nucleus}{}_{}$}%
    \dimen@\wd\tw@
    \advance\dimen@-\wd\z@
    \divide\dimen@ 3
    \@tempdima\wd\tw@
    \advance\@tempdima-\scriptspace
    \divide\@tempdima 10
    \advance\dimen@-\@tempdima
    \ifdim\dimen@>\z@ \dimen@0pt\fi
    \rel@kern{0.6}\kern-\dimen@
    \if#31
      \overline{\rel@kern{-0.6}\kern\dimen@\macc@nucleus\rel@kern{0.4}\kern\dimen@}%
      \advance\dimen@0.4\dimexpr\macc@kerna
      \let\final@kern#2%
      \ifdim\dimen@<\z@ \let\final@kern1\fi
      \if\final@kern1 \kern-\dimen@\fi
    \else
      \overline{\rel@kern{-0.6}\kern\dimen@#1}%
    \fi
  }%
  \macc@depth\@ne
  \let\math@bgroup\@empty \let\math@egroup\macc@set@skewchar
  \mathsurround\z@ \frozen@everymath{\mathgroup\macc@group\relax}%
  \macc@set@skewchar\relax
  \let\mathaccentV\macc@nested@a
  \if#31
    \macc@nested@a\relax111{#1}%
  \else
    \def\gobble@till@marker##1\endmarker{}%
    \futurelet\first@char\gobble@till@marker#1\endmarker
    \ifcat\noexpand\first@char A\else
      \def\first@char{}%
    \fi
    \macc@nested@a\relax111{\first@char}%
  \fi
  \endgroup
}
\makeatother

\newcommand{\I}{\mathcal{I}}

\def\NN{{\mathbf N}}
\def\CC{{\mathbf C}}
\def\AAA{{\mathbf A}}

\def\QQ{{\mathbf Q}}

\def\fra{{\mathfrak a}}
\def\frb{{\mathfrak b}}

\begin{document}

\vspace{\baselineskip}

\title{Hodge ideals and minimal exponents of ideals}

\author[M. Musta\c{t}\v{a}]{Mircea~Musta\c{t}\u{a}}
\address{Department of Mathematics, University of Michigan,
Ann Arbor, MI 48109, USA}
\email{{\tt mmustata@umich.edu}}

\author[M.~Popa]{Mihnea~Popa}
\address{Department of Mathematics, Northwestern University, 
2033 Sheridan Road, Evanston, IL
60208, USA} \email{{\tt mpopa@math.northwestern.edu}}

\thanks{MM was partially supported by NSF grant DMS-1701622 and a Simons Fellowship; MP was partially supported by NSF grant
DMS-1700819.}


\begin{abstract}
We define and study Hodge ideals associated to a coherent ideal sheaf $\fra$ on a smooth complex variety, via algebraic constructions based on the already existing concept of Hodge ideals associated to $\QQ$-divisors. We also define the generic minimal exponent
of $\fra$, extending the standard invariant for hypersurfaces. We relate it to Hodge ideals, and show that it is a root of the Bernstein-Sato polynomial of $\fra$.
\end{abstract}

\maketitle

\section{Introduction}

Let $X$ be a smooth complex algebraic variety. If $D$ is a reduced hypersurface in $X$ and $\shO_X(*D)$ is the sheaf of rational functions on $X$
with poles along $D$, then Saito's theory of mixed Hodge modules
\cite{Saito-MHM} endows $\shO_X(*D)$ with a Hodge filtration. This filtration can be described via a sequence of  \emph{Hodge ideals} $I_p(D)$, for $p\geq 0$, that were systematically studied in \cite{MP1}. More generally, it was shown in \cite{MP4} that one can attach Hodge ideals
to arbitrary effective $\QQ$-divisors on $X$. These invariants provide ``higher versions" of multiplier ideals, which have been playing an important role in birational geometry (see \cite[Chapter~9]{Lazarsfeld}), and which essentially correspond to the case $p =0$ in the sequence above.


Our goal in this note is to attach similar invariants to (rational powers) of an arbitrary coherent ideal $\fra$ on $X$. To this end, there are two natural approaches.  The first is based on studying the Hodge filtration on the local cohomology sheaves ${\mathcal H}_Z^q(\shO_X)$, where $Z$ is the closed subscheme 
associated to $\fra$. In this approach one stays close to Hodge theory, but the filtrations cannot be described anymore via ideals in $\shO_X$;
we plan to tackle this study in future work.
Here we take an algebraic approach, motivated by the theory of multiplier ideals, defining Hodge ideals for rational powers of coherent ideals
by making use of the existing notion for effective $\QQ$-divisors.

After replacing $X$ by the subsets in an affine open cover, we may assume that $X$ is affine and that the
ideal $\fra$ is generated by $f_1,\ldots,f_r\in\shO_X(X)$. A basic fact about multiplier ideals is that if $D$ is defined by $f=\sum_{i=1}^r\alpha_if_i$, with $\alpha_i\in\CC$ general, then for every $\lambda<1$ we have 
$$\I(\fra^{\lambda})=\I(\lambda D).$$
However, for $p\geq 1$, it turns out that even in simple examples the ideal $I_p(\lambda D)$, with $D$ as above, might depend on $D$. 

Instead, given a positive rational number $\lambda\leq 1$, we define $I_p(\fra^{\lambda})$ to be the ideal generated by all $I_p(\lambda D)$, where $D$ is the divisor defined by any $f\in \fra$ that satisfies a mild condition (for example, if $\fra$ is reduced in codimension $1$, we may take all $f\in\fra$
that define reduced divisors). We show that it is enough in fact to let $D$ vary over the divisors defined by general linear combinations of the generators of $\fra$. 
Yet another equivalent description of $I_p(\fra^{\lambda})$ is the following: if $y_1,\ldots,y_r$ denote the coordinate functions on $\AAA^r$ and we consider the 
regular function $g=\sum_{i=1}^ry_if_i$ on $X\times\AAA^r$, defining the divisor $G$, then $I_p(\fra^{\lambda})$ is generated by the coefficients of all elements of
$I_p(\lambda G)\subseteq\shO_X(X)[y_1,\ldots,y_r]$. These equivalent descriptions of $I_p(\fra^{\lambda})$ are discussed in Section 2. It is not hard to extend them to a definition in the global case.

In Section 3, we use the properties of Hodge ideals for $\QQ$-divisors proved in \cite{MP4} to show corresponding results in our more general context.
For example, we derive analogues of the Restriction Theorem and the Subadditivity Theorem in this setting. Some examples of Hodge ideals associated to ideals are computed in Section 4. 

We note that this theory of Hodge ideals associated to ideal sheaves is not yet completely satisfactory, since some of the main tools from the study of Hodge ideals of divisors are still missing. The main reason is the lack of a direct connection with Hodge theory. For example, we don't know whether on projective varieties there is a vanishing theorem for Hodge ideals associated to an ideal $\fra$ (see Question~\ref{vanishing}).

Finally, in Section 5 we define and study an extension of the notion of \emph{minimal exponent} to the case of ideals. 
Recall first that for a divisor $D$ and $x\in {\rm Supp}(D)$, the minimal exponent $\widetilde{\alpha}_x(D)$ is the negative of the largest root 
of the reduced Bernstein-Sato polynomial of $D$ at $x$. This is a refined version of the log canonical threshold ${\rm lct}_x(D)$, which is equal to $\min\{\widetilde{\alpha}_x(D),1\}$. It is intimately linked to Hodge ideals as follows: by \cite[Corollary~C]{MP3}, if $D$ is a reduced divisor and $\lambda$ is a rational number with $0<\lambda\leq 1$, then for every $p$ we have $I_p(\lambda D)_x=\shO_{X,x}$ if and only if $p+\lambda\leq \widetilde{\alpha}_x(D)$.

For an arbitrary ideal $\fra$, and a point $x$ in the zero-locus of $\fra$, we define an invariant, the \emph{generic minimal exponent}
$\overline{\alpha}_x(\fra)$, which is the minimal exponent at $x$ of a general hypersurface containing the subscheme defined by $\fra$. More precisely,
if $D$ is the divisor defined by a general linear combination of generators of $\fra$ in an affine open neighborhood of $x$, then 
$\overline{\alpha}_x(\fra)=\widetilde{\alpha}_x(D)$. As in the divisorial case, if $\lambda$ is a rational number with $0<\lambda\leq 1$,  and $\fra$ is radical in codimension $1$ around $x$, then
$$I_p(\fra^{\lambda})_x=\shO_{X,x} \iff p+\lambda\leq\overline{\alpha}_x(\fra).$$ 
(If $\fra$ is not radical  in codimension $1$ around $x$,
then $\overline{\alpha}_{\fra,x}$ is equal to the log canonical threshold ${\rm lct}_x(\fra)$ of $\fra$ at $x$.)
We extend the basic properties of minimal exponents of divisors to the case of arbitrary ideals. The main result we prove, 
Theorem \ref{root}, states that $\overline{\alpha}_x(\fra)$ is a root of the Bernstein-Sato polynomial $b_{\fra}(s)$ defined in \cite{BMS}.

\section{Equivalent definitions}

Our goal in this section is to give the definition of Hodge ideals associated to arbitrary nonzero ideals and provide some equivalent descriptions. Let $X$ be a smooth $n$-dimensional complex algebraic variety and
$\fra$ a nonzero coherent ideal on $X$.

Since $X$ is smooth, after taking a suitable affine open cover of $X$, we reduce to the case 
 when $X$ is an affine variety and ${\mathfrak a}=g\cdot {\mathfrak b}$, where ${\mathfrak b}$ is an ideal
defining a closed subscheme of codimension $\geq 2$. From now on, unless explicitly mentioned otherwise, 
we make this assumption.  We write ${\rm div}(\fra)$ for the effective divisor defined by $g$
(this is independent of the choice of $g$, and moreover, it can be associated to $\fra$ on any smooth variety).
Note that if $h_1,\ldots,h_r$ are generators of $\frb$ and $\alpha_1,\ldots,\alpha_r\in\CC$ are general,
then $\sum_i\alpha_ih_i$ defines a reduced effective divisor on $X$, without any common components with ${\rm div}(\fra)$.


\begin{definition}\label{defi1}
If $X$ is a smooth affine variety and ${\mathfrak a}=g\cdot {\mathfrak b}$ as above, then
for every $p\geq 0$ and $\lambda\in (0,1]\cap\QQ$, the 
$p$th \emph{Hodge ideal} of ${\mathfrak a}^{\lambda}$ is
$$I_p({\mathfrak a}^{\lambda}):=\sum_EI_p\big(\lambda ({\rm div}({\mathfrak a})+E)\big),$$
where the sum is over all reduced effective divisors $E$, defined by elements $h\in {\mathfrak b}$,
and which have no common components with ${\rm div}({\mathfrak a})$. Equivalently, we have
$$I_p({\mathfrak a}^{\lambda}):=\sum_DI_p(\lambda D),$$
where $D$ varies over the divisors defined by elements of ${\mathfrak a}$, such that
$D-{\rm div}({\mathfrak a})$ is reduced, without common components with ${\rm div}({\mathfrak a})$.
\end{definition}

This definition makes sense for $\lambda > 1$ as well. However, we believe that from the point of view we want to adopt it does not give the ``correct" objects; see for instance Remark \ref{big_lambda} below. We prefer thus to restrict to $\lambda \in (0,1]$.

\begin{remark}[Reduced subschemes]\label{def_reduced}
Note that if ${\mathfrak a}$ defines a subscheme that is reduced in codimension 1, then
$$I_p(\fra^{\lambda}):=\sum_DI_p(\lambda D),$$
where the sum is over all reduced effective divisors $D$ defined by elements of $\fra$.
\end{remark}

\begin{remark}[Principal ideals]
If the ideal $\mathfrak a$ is principal, defining a divisor $D$, then $I_p({\mathfrak a}^{\lambda})=I_p(\lambda D)$
(in which case, if $D={\rm div}(g)$, we also denote this by $I_p(g^{\lambda})$). This follows from the fact
that if $E$ is an effective divisor, with ${\rm Supp}(D)$ and ${\rm Supp}(E)$ having no common components, then 
$I_p\big(\lambda(D+E)\big)\subseteq I_p(\lambda D)$. This is a consequence of the Subadditivity Theorem for Hodge ideals  (see \cite[Theorem~15.1]{MP4}).
\end{remark}

Before giving other equivalent descriptions of $I_p(\fra^{\lambda})$, we introduce some notation.
Suppose that $X={\rm Spec}(R)$ is affine and 
$J\subseteq R[y_1,\ldots,y_r]$, for some $r\geq 1$, is an ideal. We define the ideal 
${\rm Coeff}(J)$ of $R$ as follows. Choose generators
$Q_1,\ldots,Q_s$ for $J$ and write each of them as
$$Q_i=\sum_{u \in\NN^r}P_{u,i}y^u,$$
with $P_{u,i}\in R$ and $y^u=y_1^{u_1}\cdots y_r^{u_r}$ for every $u=(u_1,\ldots,u_r)\in\NN^r$
(here $\NN$ is the set of nonnegative integers). We then put
$${\rm Coeff}(J):=(P_{u,i}\mid u\in\NN^r,1\leq i\leq s)\subseteq R.$$
Note that if $Q=\sum_{i=1}^s h_iQ_i$ is in $J$, and if 
$$h_i=\sum_{u\in\NN^r}c_{u,i}y^u,$$
then 
$$Q=\sum_{u\in\NN^r}\left(\sum_{i=1}^s\sum_{v+w=u}c_{v,i}P_{w,i}\right)y^u$$
and 
$$\sum_{i=1}^s\sum_{v+w=u}c_{v,i}P_{w,i}\in (P_{u,j}\mid u\in\NN^r,1\leq j\leq s).$$
Therefore the definition of ${\rm Coeff}(J)$ is independent of the choice of generators for $J$.

\begin{lemma}\label{lem1}
If $J=(Q_1,\ldots,Q_s)$ is an ideal in $R[y_1,\ldots,y_r]$, then 
the ideal ${\rm Coeff}(J)$ is generated by $\{Q_1(\alpha),\ldots,Q_s(\alpha)\mid\alpha\in\CC^r\}$.
Moreover, given any non-empty open subset $U\subseteq\CC^r$, it is enough to only consider those $\alpha\in U$. 
\end{lemma}

\begin{proof}
Note that if $P\in R[y_1,\ldots,y_r]$ has degree $d$ and for $j\in \Gamma$, with $|\Gamma|\geq d+1$, we consider
$$\alpha^{(j)}=(\alpha^{(j)}_1,\ldots,\alpha^{(j)}_r)\in\CC^r$$
such that $\alpha_i^{(j)}\neq\alpha_i^{(j')}$ for all $i$ and all $j\neq j'$ in $\Gamma$, then the coefficients of $P$ lie in the ideal generated by
$\{P(\alpha^{(j)})\mid j\in \Gamma\}$. (This follows by induction on $r$ from the formula for the determinant of the Vandermonde matrix.) The assertions in the lemma are an immediate consequence. 
\end{proof}

We can now give two other descriptions of $I_p({\mathfrak a}^\lambda)$. As before, we assume that $X={\rm Spec}(R)$ is smooth and affine
and that we can write $\fra=g\cdot\frb$, with $\frb$ defining a subscheme of codimension $\geq 2$.

\begin{theorem}\label{prop1}
With the above notation, if
$f_1,\ldots,f_r$ are generators of ${\mathfrak a}$, then for every $p\geq 0$ and $\lambda \in (0,1]\cap\QQ$ the following hold:
\begin{enumerate}
\item[i)] $I_p({\mathfrak a}^{\lambda})$ is generated by the ideals $I_p(\lambda D)$, where $D$ is the divisor of a general linear combination $\sum_i\alpha_if_i$, with $\alpha_i\in\CC$.
\item[ii)] We have 
$$I_p({\mathfrak a}^{\lambda})={\rm Coeff}\big(I_p(\lambda G)\big),$$ with
$G$ being the divisor on $X\times\AAA^r$ defined by $\sum_{i=1}^ry_if_i$, where 
$y_1,\ldots,y_n$ are the coordinates on $\AAA^r$. 
\end{enumerate}
\end{theorem}

\begin{remark}\label{rmk_prop1}
By assumption, we can write $f_i=gh_i$, where $h_1,\ldots,h_r$ are generators for the ideal ${\mathfrak b}$. 
Note that the divisor $G$ in ii) can be written as ${\rm pr}_1^*\big({\rm div}({\mathfrak a})\big)+G'$, where $G'$ is a reduced divisor having no common components with ${\rm pr}_1^*\big({\rm div}({\mathfrak a})\big)$. 
Indeed, $G$ is defined by $g\cdot\sum_{i=1}^ry_ih_i$ and we let $G'$ be the divisor defined by $\sum_{i=1}^ry_ih_i$.
If $T$ is an irreducible component of $G'$, which either appears with multiplicity $\geq 2$ in $G'$, or is also a component of ${\rm pr}_1^*\big({\rm div}({\mathfrak a})\big)$, 
then $T$ is the pull-back of a prime divisor on $X$. (In the first case, this follows from the fact that for general 
$\alpha_1,\ldots,\alpha_r\in\CC$, the element $\sum_i\lambda_ih_i\in R$ defines
a reduced divisor on $X$.) After replacing $X$ by a suitable affine open subset, we may assume that $T$ is defined by
$h\in R$ such that $h$ divides $h_i$ for all $i$.
This contradicts the fact that 
 ${\mathfrak b}$ defines a subscheme of codimension $\geq 2$.
\end{remark}

\begin{proof}[Proof of Theorem~\ref{prop1}]
Let us denote by $I'_p({\mathfrak a}^{\lambda})$ the ideal generated by the $I_p(\lambda D)$, 
where $D$ is the divisor defined by a general linear combination
$\sum_i\alpha_if_i$. 
Let $Q_1,\ldots,Q_s\in R[y_1,\ldots,y_r]$ be generators for $I_p(\lambda G)$. 
We write $G={\rm pr}_1^*\big({\rm div}({\mathfrak a})\big)+G'$ as in Remark~\ref{rmk_prop1}.
For every $\alpha=(\alpha_1,\ldots,\alpha_r)\in\CC^r$, 
 the restriction $G_{\alpha}$ of $G$ to 
$$X\simeq X\times\{\alpha\}\hookrightarrow X\times {\mathbf A}^r$$
 is equal to 
the sum of ${\rm div}({\mathfrak a})$ and the divisor $G'_{\alpha}$ defined on $X$ by $\sum_{i=1}^r\alpha_ih_i$. 
Note that we have $G_{\rm red}={\rm pr}_1^*\big({\rm div}({\mathfrak a})\big)_{\rm red}+G'$. 
If $G'_{\alpha}$ is reduced, having no common components with ${\rm div}({\mathfrak a})$, then 
the restriction of $G_{\rm red}$ to 
$X\simeq X\times\{\alpha\}$ is equal to ${\rm div}({\mathfrak a})_{\rm red}+G'_{\alpha}=(G_{\alpha})_{\rm red}$. 
In this case we can apply 
the Restriction Theorem for Hodge ideals (see \cite[Theorem~13.1]{MP4}) 
to deduce that for such $\alpha$, we have
$$I_p\big(\lambda\cdot {\rm div}({\mathfrak a})+\lambda\cdot G'_{\alpha}\big)\subseteq I_p(\lambda G)\cdot \shO_X=\big(Q_1(\alpha),\ldots,Q_r(\alpha)\big).$$
Moreover, this is an equality for general $\alpha$. 

The fact that ${\rm Coeff}\big(I_p (\lambda G)\big)=I'_p({\mathfrak a}^{\lambda})$ now follows from Lemma~\ref{lem1}.
Moreover, it is clear by definition that
$I'_p({\mathfrak a}^{\lambda})\subseteq I_p({\mathfrak a}^{\lambda})$. The above consequence of the Restriction Theorem gives the inclusion $I_p({\mathfrak a}^{\lambda})\subseteq {\rm Coeff}\big(I_p(\lambda G)\big)$, completing the proof of the result.
\end{proof}

\begin{remark}\label{rmk_open_restriction}
If $X$ and ${\mathfrak a}$ are as in Theorem~\ref{prop1} and $U$ is an affine open subset of $X$, then 
it follows from either of the two descriptions of $I_p({\mathfrak a}^{\lambda})$ in  the theorem that the restriction of $I_p({\mathfrak a}^{\lambda})$ to $U$
is $I_p\big(({\mathfrak a}\vert_U)^{\lambda}\big)$. We may thus define $I_p({\mathfrak a}^\lambda)$ by gluing the objects defined locally in a suitable affine open cover.
\end{remark}

\begin{definition}[Global definition]
If $X$ is a smooth variety, ${\mathfrak a}$ is a nonzero ideal on $X$, and $\lambda\in (0,1]\cap\QQ$, 
we choose an affine open cover of $X$ such that on each open subset $U$ belonging to the cover, we can factor 
$\fra\vert_U$ as $g_U\cdot\frb_U$,
with $\frb_U$ defining a subscheme of codimension $\geq 2$. For every such $U$ we may thus define $I_p\big((\fra\vert_U)^{\lambda}\big)$, and it follows from Remark~\ref{rmk_open_restriction} that these ideals glue to give an ideal $I_p(\fra^{\lambda})$ on $X$. This is clearly independent 
of the choice of cover.
\end{definition}


\begin{remark}[Mixed Hodge ideals]\label{mixed_HI}
Suppose that we have nonzero ideals 
$\fra_1,\ldots,\fra_r$ on $X$.
We may assume that $X$ is affine, and for each $i$ we can factor $\fra_i=g_i\cdot\frb_i$, where $\frb_i$ defines a subscheme of codimension $\geq 2$. 
For rational numbers $\lambda_1,\ldots,\lambda_r\in (0,1]$, we consider divisors $D=\sum_{i=1}^r\lambda_i\big({\rm div}(\fra_i)+E_i\big)$, 
where each $E_i$ is defined by an element of $\frb_i$, such that $\sum_iE_i$ is a reduced divisor that has no common components with $\sum_i{\rm div}(\fra_i)$. This allows us, as in Definition \ref{defi1},  to define an ideal 
$$I_p (\fra_1^{\lambda_1} \cdots \fra_r^{\lambda_r}) \subseteq \shO_X.$$
There is an analogue of Theorem~\ref{prop1} in this more general setting and the interested reader will have no trouble stating and proving it.
\end{remark}

\section{Basic properties}

In this section we extend some basic properties of Hodge ideals from the case of divisors to that of ideals. 

\begin{proposition}\label{behaviour_under_inclusion}
If ${\mathfrak a}\subseteq {\mathfrak b}$ are nonzero ideals on the smooth variety $X$, such that the divisors 
${\rm div}({\mathfrak a})-{\rm div}({\mathfrak b})$
and ${\rm div}({\mathfrak b})$ have no common components, then for every $p\geq 0$ and $\lambda\in (0,1]\cap\QQ$ we have
$$I_p({\mathfrak a}^{\lambda})\subseteq I_p({\mathfrak b}^{\lambda}).$$
\end{proposition}

\begin{proof}
We may assume that $X$ is affine and that ${\mathfrak a}$ is generated by $f_1,\ldots,f_r$ and ${\mathfrak b}$
is generated by $f_1,\ldots,f_r,f_{r+1},\ldots,f_{r+s}$. Furthermore, we may assume that for $i\leq r$
we can write $f_i=gh_i$ such that $h_1,\ldots,h_r$ define a subscheme of codimension $\geq 2$ and similarly, for $i\leq r+s$ we can write $f_i=g'h'_i$
such that $h'_1,\ldots,h'_{r+s}$ define a subscheme of codimension $\geq 2$. We can then write $g=g'u$, for some 
$u\in\shO_X(X)$. 

Consider $f=gh$, where $h\in (h_1,\ldots,h_r)$ defines a reduced divisor without common components with the divisor ${\rm div}({\mathfrak a})$ defined by $g$. 
Since we can write $f=g'(uh)$, and ${\rm div}(uh)={\rm div}(u)+{\rm div}(h)$ has no common components with 
${\rm div}(g')$
(note that by hypothesis, ${\rm div}(u)$ and ${\rm div}(g')$ have no common components), it follows from the definition that
$$I_p\big(\lambda\cdot{\rm div}(f)\big)\subseteq I_p({\mathfrak b}^{\lambda}).$$
Since this holds for all $f$ as above, we obtain the assertion in the proposition.
\end{proof}

\begin{remark}
The condition on ${\rm div}(\fra)$ and ${\rm div}(\frb)$ in Proposition~\ref{behaviour_under_inclusion}
cannot be dropped: if $D$ and $E$ are effective $\QQ$-divisors such that $D-E$ is effective, it is not the case that we always have
$I_p(D)\subseteq I_p(E)$. In fact, this can fail even when $D$ and $E$ are rational multiples of the same integral divisor, see \cite[Example~10.5]{MP4}.
\end{remark}

We next show that $I_0$ coincides with a multiplier ideal.

\begin{proposition}\label{rem1}
If $X$ is a smooth variety and $\fra$ is a nonzero ideal on $X$, then for every $\lambda\in (0,1]\cap\QQ$ we have 
$$
I_0(\fra^{\lambda})=\I \big({\mathfrak a}^{\lambda-\epsilon}\big)\quad\text{for} \quad 0<\epsilon\ll 1.
$$
\end{proposition}

\begin{proof}
It is enough to check this when $X$ is affine.
If $h$ is a general linear combination of a system of generators of ${\mathfrak a}$, then it follows from
\cite[Proposition~9.2.28]{Lazarsfeld} that
$$\I\big(h^{\lambda-\epsilon}\big)=\I\big({\mathfrak a}^{\lambda-\epsilon}\big).$$
If $D$ is the divisor defined by $h$, then
$$\I\big(h^{\lambda-\epsilon}\big)=I_0(\lambda D)$$
by \cite[Proposition~9.1]{MP4}. 
The assertion now follows from Theorem~\ref{prop1}i).
\end{proof}

\begin{remark}\label{big_lambda}
Note that if we also allowed $\lambda > 1$ in Definition \ref{defi1} and Theorem \ref{prop1}, using the fact that $I_0 \big((\alpha + 1) D\big) = \shO_X(-D) \cdot I_0 (\alpha D)$ for every $\alpha \in \QQ$, we would get $I_0 (\fra^{\alpha + 1}) = \fra \cdot I_0 (\fra^{\alpha})$, and not $\I \big({\mathfrak a}^{\alpha + 1 -\epsilon}\big)$.
\end{remark}

\begin{proposition}\label{smooth_pull_back}
If $\fra$ is a nonzero ideal on the smooth variety $X$ and $\varphi\colon Y\to X$ is a smooth morphism, then for every $\lambda\in (0,1]\cap\QQ$ and every $p\geq 0$, we have
$$I_p(\fra^{\lambda})\cdot\shO_{Y}= I_p\big((\fra\cdot\shO_Y)^{\lambda}\big).$$
\end{proposition}

\begin{proof}
We may clearly assume that both $X$ and $Y$ are affine, and let $f_1,\ldots,f_r$ be generators of ${\mathfrak a}$. This implies that $f_1\circ \varphi,\ldots,f_r\circ\varphi_r$ generate ${\mathfrak a}\cdot\shO_Y$. 
If $\alpha =(\alpha_1,\ldots,\alpha_r)\in\CC^r$
is general and $D_{\alpha}$ is defined by $\sum_i\alpha_if_i$, then $\varphi^*D_{\alpha}$ is defined by 
$\sum_i\alpha_i(f_i\circ\varphi)$.
Since $I_p({\mathfrak a}^\lambda)$ is generated by such $I_p(\lambda D_{\alpha})$ and 
$I_p\big((\fra\cdot\shO_Y)^{\lambda}\big)$ is generated by such $I_k(\lambda \varphi^*D_{\alpha})$, we obtain the assertion in the proposition thanks to the lemma below. 
\end{proof}

The following is the extension of \cite[Proposition~15.1]{MP1} to the case of $\QQ$-divisors; it is stated only
implicitly in \cite{MP4}.

\begin{lemma}
If $\varphi \colon Y \to X$ is a smooth morphism of smooth varieties, and $D$ is an effective $\QQ$-divisor on 
$X$, then for every $p \ge 0$ we have 
$$I_p (\varphi^* D) = I_p (D) \cdot \shO_Y.$$
\end{lemma}
\begin{proof}
By possibly shrinking $X$, we may assume that $D = \alpha H$, where $\alpha$ is a positive rational number and $H$ is the effective Cartier divisor  defined by a function $h \in \shO_X(X)$. We denote by $Z$ the support
of $D$. We then have that $\varphi^* D = \alpha \varphi^* H$, and $\varphi^* H$ is defined by $h' = h \circ \varphi$. Moreover, since $\varphi$ is smooth, the divisor $Z' = \varphi^* Z$ is reduced, and is therefore equal to the support of $\varphi^* D$.

Note now that, in the notation of \cite[\S2 and \S4]{MP4}, the Hodge ideal $I_p (D)$ is defined by the Hodge filtration on the $\Dmod_X$-module $\Mmod (h^{-\alpha})$, in the sense that
$$F_p \Mmod (h^{-\alpha}) = I_p (D) \otimes \shO_X (pZ) \cdot h^{-\alpha}.$$
(Cf. more precisely \cite[Remark~4.3]{MP4}.) Analogously, we have 
$$F_p \Mmod ({h'}^{-\alpha}) = I_p (\varphi^*D) \otimes \shO_X (pZ') \cdot {h'}^{-\alpha}.$$
It suffices then to have 
$$F_p \Mmod ({h'}^{-\alpha}) = \varphi^* F_p \Mmod (h^{-\alpha}),$$
which is deduced in \cite[Remark~2.15]{MP4} as a consequence of the behavior of mixed Hodge modules
under smooth morphisms and base-change.
\end{proof}

\medskip

\begin{proposition}\label{prop_basic}
Let $X$ be a smooth complex variety and ${\mathfrak a}, {\mathfrak b}$ nonzero ideals on $X$.
\begin{enumerate}
\item[i)] For every $p\geq 0$ and every $\lambda\in (0,1]\cap\QQ$, we have
$${\mathfrak a}^p\cdot I_0({\mathfrak a}^{\lambda})\subseteq I_p({\mathfrak a}^{\lambda}).$$ 
\item[ii)] For every $p\geq 0$ and every $\lambda\in (0,1]\cap\QQ$, we have
$${\mathfrak a}^{p+1}\cdot I_p({\mathfrak b}^{\lambda})\subseteq I_p({\mathfrak a} {\mathfrak b}^{\lambda}).$$
\end{enumerate}
\end{proposition}

\begin{proof}
In order to prove the inclusion in i),
we may assume that $X$ is affine. Let $h_1,\ldots,h_r$ be generators of ${\mathfrak a}$. 
If $h=\sum_{j=1}^r\alpha_j h_j$, with $\alpha_1,\ldots,\alpha_r\in\CC$ general, then as in the 
proof of Proposition~\ref{rem1} we have $I_0(h^{\lambda})=I_0({\mathfrak a}^{\lambda})$.
Using \cite[Remark~4.2]{MP4}, we have
$$h^p \cdot I_0(\fra^{\lambda})=h^p \cdot I_0(h^{\lambda})\subseteq  I_p(h^{\lambda})\subseteq I_p({\mathfrak a}^{\lambda}).$$
Since this holds for every $\lambda_1,\ldots,\lambda_r\in \CC$ general, and we are in characteristic $0$, we conclude that 
$${\mathfrak a}^p\cdot I_0({\mathfrak a}^{\lambda})\subseteq
I_p({\mathfrak a}^{\lambda}).$$

We next prove ii). 
Consider first the case when ${\mathfrak a}=(f)$ and ${\mathfrak b}=(g)$ are principal ideals.
In this case we have an inclusion of filtered $\Dmod_X$-modules
\begin{equation}\label{eq_prop_basic}
\shO_X[1/g]g^{-{\lambda}}\subseteq \shO_X[1/fg]g^{-\lambda}.
\end{equation}
For the definition of these $\Dmod_X$-modules, which play an important role in defining Hodge ideals of 
$\QQ$-divisors, we refer to \cite[\S2]{MP4}. Recall from \cite[\S4]{MP4} that by the definition of Hodge ideals, 
we have
$$F_p\shO_X[1/g]g^{-{\lambda}}=I_p(g^{\lambda})\cdot\shO_X\big(p\cdot {\rm div}(g)_{\rm red}\big)g^{-\lambda}$$
and 
$$F_p\shO_X[1/fg]g^{-{\lambda}}=I_p(fg^{\lambda})\cdot\shO_X\big(p\cdot {\rm div}(fg)_{\rm red}\big)f^{-1}g^{-\lambda}.$$
By passing to filtered pieces, the inclusion (\ref{eq_prop_basic}) thus gives
$$\shO_X\big(p\cdot {\rm div}(g)_{\rm red}\big)\cdot I_p(g^{\lambda})g^{-\lambda}\subseteq 
\shO_X\big(p\cdot {\rm div}(fg)_{\rm red}\big)\cdot I_p(fg^{\lambda})f^{-1}g^{-\lambda},$$
hence
$$f^{p+1}I_p(g^{\lambda})\subseteq I_p(fg^{\lambda}).$$

We now turn to the case of arbitrary ideals.
We may and will assume that $X$ is affine, with $R=\shO_X(X)$, and that we have factorizations $\fra=\varphi\cdot\fra'$ and $\frb=\psi\cdot\frb'$,
with $\fra'$ and $\frb'$ defining subschemes of codimension $\geq 2$.
Consider a general linear combination $g$ of generators of $\frb$, that defines a divisor $E$. By the generality condition, we may assume that
$E-{\rm div}(\frb)$ is reduced, without any common components with ${\rm div}(\fra)+{\rm div}(\frb)$.
In this case the divisors
$${\rm div}({\mathfrak a}g)-{\rm div}({\mathfrak a}{\mathfrak b})=E-{\rm div}({\mathfrak b})$$
and ${\rm div}(\fra\frb)$ have no components in common, hence the obivous analogue of 
Proposition~\ref{behaviour_under_inclusion} for mixed Hodge ideals 
(see Remark \ref{mixed_HI}) gives
$$I_p({\mathfrak a}g^{\lambda})\subseteq 
I_p({\mathfrak a}{\mathfrak b}^{\lambda}).$$
Using the characterization of $I_p(\frb^{\lambda})$ in Theorem~\ref{prop1}i), it then suffices to show that for every $g$ as above we have 
\begin{equation}\label{eq2_prop_basic}
{\mathfrak a}^{p+1}\cdot I_p(g^{\lambda})\subseteq I_p({\mathfrak a}g^{\lambda}).
\end{equation}
Let $f_1,\ldots,f_r$ be generators of ${\mathfrak a}$ and consider $h=\sum_{i=1}^rf_iy_i\in R[y_1,\ldots,y_r]$. 
It follows from the case of principal ideals that 
$$h^{p+1}I_p(g^{\lambda})\subseteq I_p(hg^{\lambda}).$$
(Note that $I_p(g^{\lambda})\cdot R[y_1,\ldots,y_r]$ is the $p^{\rm th}$ Hodge ideal with exponent $\lambda$
for the image of $g$ in $R[y_1,\ldots,y_r]$, by Proposition~\ref{smooth_pull_back}.)
This implies
$${\rm Coeff}(h^{p+1})\cdot I_p(g^{\lambda})={\rm Coeff}\big(h^{p+1}I_p(g^{\lambda})\big)\subseteq {\rm Coeff}\big(I_p(hg^{\lambda})\big)=I_p({\mathfrak a}g^{\lambda}),$$
where the last equality follows from the analogue of Theorem~\ref{prop1}ii) for mixed Hodge ideals.
On the other hand, it follows from the definition that
$${\rm Coeff}(h^{p+1})={\mathfrak a}^{p+1},$$
and we obtain the inclusion in (\ref{eq2_prop_basic}).
\end{proof}

For exponent $\lambda =1$ and ideals that are radical in codimension $1$, Hodge ideals become deeper as 
$p$ increases:

\begin{proposition}\label{prop_chain_of_ideals}
If $X$ is a smooth variety and ${\mathfrak a}$ is a nonzero ideal that is radical in codimension $1$, then
$$I_{p+1}({\mathfrak a})\subseteq I_p({\mathfrak a})$$
for every nonnegative integer $p$.
\end{proposition}

\begin{proof}
We may assume that $X$ is affine and ${\mathfrak a}$ is generated by $h_1,\ldots,h_r$. If $D$ is a reduced divisor defined by 
$\sum_{i=1}^r\alpha_ih_i$, then
$$I_{p+1}(D) \subseteq I_p(D)$$ 
by \cite[Proposition~13.1]{MP1}. The assertion in the proposition now follows from the definition of $I_p({\mathfrak a})$ and $I_{p+1}({\mathfrak a})$; see also Remark \ref{def_reduced}.
\end{proof}

For arbitrary $\lambda$ we only have the following:

\begin{proposition}\label{inclusion_mod}
Let $X$ be a smooth variety and consider a nonzero ideal $\fra$ on $X$ which is radical in codimension $1$. If $p$ and $p'$  are nonnegative integers and 
$\lambda, \lambda^\prime \in \QQ\cap (0,1]$ are such that $p + \lambda  \le p' +\lambda^\prime$, then
$$I _{p'} ({\mathfrak a}^{\lambda^\prime}) \subseteq I_p({\mathfrak a}^\lambda) \,\,\,\,\,\,{\rm mod}\,\,\,\,\fra,$$
i.e. the inclusion holds in the quotient $\shO_X/\fra$.
\end{proposition}
\begin{proof}
We may assume that $X$ is affine, and let $D$ be the divisor corresponding to a general linear combination $f$
of generators of $\fra$. By \cite[Theorem~A' and Remark 4.8]{MP3}, we have 
$$I _{p'} (\lambda^\prime D) \subseteq I_p( \lambda D) \,\,\,\,\,\,{\rm mod}\,\,\,\,f,$$
and hence also mod $\fra$. Indeed, mod $f$ these statements say that $I_p (\lambda D)$ coincides with $\widetilde{V}^{p+\lambda} \shO_X$, Saito's microlocal $V$-filtration on $\shO_X$ along $f$; this is a decreasing filtration. We can then use Theorem \ref{prop1}i) to conclude.
\end{proof}

We now turn to the analogue of the Restriction Theorem for multiplier ideals (cf. \cite[Theorem~9.5.1 and Example~9.5.4]{Lazarsfeld}) and for Hodge ideals of divisors (cf. \cite[Theorem~A]{MP2} and \cite[Theorem~13.1]{MP4}).
Let $X$ be a smooth complex variety and $H\subseteq X$ a smooth, irreducible hypersurface. Consider an ideal $\fra$ on $X$ such that
$\fra_H:=\fra\cdot\shO_H$ is nonzero. We define on $H$ the divisor $F=\sum_Ta_TT$, where $T$ varies over
the components of ${\rm div}(\fra_H)$  and 
$$a_T :={\rm ord}_T\big({\rm div}(\fra)_{\rm red}\vert_H\big)+{\rm ord}_T(\fra_H)-{\rm ord}_T\big({\rm div}(\fra)\vert_H\big)-1.$$
It is easy to see that $a_T\geq 0$, but this will also be clear from the proof of the next theorem.

\begin{theorem}\label{prop_restriction}
With the above notation, for every $p\geq 0$ and every $\lambda\in (0,1]\cap\QQ$, we have 
\begin{equation}\label{eq_prop_restriction}
\shO_H\big(-pF)\cdot I_p(\fra_H^{\lambda})\subseteq 
I_p({\mathfrak a}^{\lambda})\cdot\shO_H.
\end{equation}
Moreover, if $H$ is sufficiently general
(for example, a general member of a basepoint-free linear system), then $F=0$ and the inclusion in (\ref{eq_prop_restriction})
is an equality.
\end{theorem}

\begin{proof}
We may assume that $X$ is affine and ${\mathfrak a}$ is generated by $h_1,\ldots,h_r$. If $\alpha_1,\ldots,\alpha_r\in\CC$ are general and $D$ is defined by
$\sum_i\alpha_ih_i$, then $D\vert_H$ is defined by a general linear combination of a system of generators of 
$\fra_H$. We can write
$D={\rm div}(\fra)+B$, with $B$ reduced and having no common components with ${\rm div}(\fra)$. Therefore we have
$$Z:=D_{\rm red}={\rm div}(\fra)_{\rm red}+B.$$ 
If $Z_H=Z\vert_H$ and $Z'_H=(Z_H)_{\rm red}$, it follows from 
\cite[Theorem~13.1]{MP4} that we have
\begin{equation}\label{eq2_prop_restriction}
\shO_H\big(-p(Z_H-Z'_H)\big)\cdot I_p(\lambda D\vert_H)\subseteq I_p(\lambda D)\cdot\shO_H.
\end{equation}
Moreover, if $H$ is sufficiently general (depending on $D$), then $Z_H=Z'_H$ and we have equality in (\ref{eq2_prop_restriction}). 

Note now that if $T$ is a prime divisor on $H$ such that ${\rm ord}_T(Z_H)\geq 2$, then ${\rm ord}_T(D\vert_H)\geq 2$, hence
$T$ is a component of ${\rm div}(\fra_H)$. In particular, there are only finitely many such $T$, independently of our choice of $D$.
Since $D$ is general, for every component $T$ of  ${\rm div}(\fra_H)$, we have
${\rm ord}_T(D\vert_H)={\rm ord}_T(\fra_H)$, hence
$${\rm ord}_T(Z_H-Z'_H)={\rm ord}_T\big({\rm div}(\fra)_{\rm red}\vert_H\big)+{\rm ord}_T(B\vert_H)-1$$
$$={\rm ord}_T\big({\rm div}(\fra)_{\rm red}\vert_H\big)+{\rm ord}_T(D\vert_H)-{\rm ord}_T\big({\rm div}(\fra)\vert_H\big)-1$$
$$={\rm ord}_T\big({\rm div}(\fra)_{\rm red}\vert_H\big)+{\rm ord}_T(\fra_H)-{\rm ord}_T\big({\rm div}(\fra)\vert_H\big)-1.$$
This shows that $Z_H-Z'_H=F$. By letting $D$ vary and using Theorem~\ref{prop1}i), we deduce from (\ref{eq2_prop_restriction})
the first assertion of the proposition.

Let us now choose divisors $D_1,\ldots,D_s$ as above such that
$$I_p(\fra^{\lambda})=\sum_{i=1}^s I_p(\lambda D_i)\quad\text{and}\quad I_p(\fra_H^{\lambda})=\sum_{i=1}^sI_p(\lambda D_i\vert_H).$$
If we take $H$ general with respect to all $D_i$, then we see that 
$$I_p(\lambda D_i\vert_H)=I_p(\lambda D_i)\cdot\shO_H\quad\text{for}\quad 1\leq i\leq s.$$
We thus obtain the second assertion of the proposition. 
\end{proof}

\begin{remark}\label{rmk_special_case_restriction}
With the notation in Theorem~\ref{prop_restriction}, if $\fra\cdot\shO_H$ is radical in codimension $1$, then $F=0$, and we get
$$I_p(\fra_H^{\lambda})\subseteq 
I_p({\mathfrak a}^{\lambda})\cdot\shO_H\quad\text{for every}\quad p\geq 0.$$
Indeed, if $X$ is affine and $D$ is as in the proof of the proposition, then the hypothesis implies that $D\vert_H$ is reduced.
In particular, we have $Z_H=Z'_H$, hence $F=0$. 
\end{remark}

\begin{corollary}\label{cor_pull_back}
Let $\varphi\colon W\to X$ be any morphism of smooth complex varieties. If ${\mathfrak a}$ is an ideal on $X$ such that
$\fra_W:={\mathfrak a}\cdot \shO_W$ is nonzero and radical in codimension $1$, then for every $p\geq 0$ and every $\lambda\in\QQ\cap (0,1]$ we have
$$I_p(\fra_W^{\lambda})\subseteq I_p(\fra^{\lambda})\cdot\shO_W.$$
\end{corollary}

\begin{proof}
We can factor $\varphi$ as 
$$W\overset{j}\hookrightarrow W\times X\overset{p}\longrightarrow X,$$ where $p$ is the projection and $j$ is a closed embedding. Since $p$ is smooth, we have
$$I_p\big(({\mathfrak a}\cdot \shO_{W\times X})^{\lambda}\big)=I_p({\mathfrak a}^{\lambda})\cdot\shO_{W\times X}$$
by Proposition~\ref{smooth_pull_back},
hence in order to prove the corollary it is enough to treat the case when $\varphi$ is a closed embedding. In this case the statement follows by an easy induction
on the codimension of $W$, using Theorem~\ref{prop_restriction} (see also Remark~\ref{rmk_special_case_restriction}).
\end{proof}

We deduce the following analogue of the Subadditivity Theorem for multiplier ideals (cf. \cite[Theorem~9.5.20]{Lazarsfeld}) and for Hodge ideals of divisors (cf. \cite[Theorem~B]{MP2} and \cite[Theorem~15.1]{MP4}).

\begin{proposition}\label{subadditivity}
If $X$ is a smooth, complex algebraic variety, and $\fra$ and $\frb$ are nonzero
ideals on $X$ such that ${\rm div}(\fra)$ and ${\rm div}(\frb)$ have no common components,
then for every nonnegative integer $p$ and every $\lambda\in \QQ\cap (0,1]$, we have 
$$I_p\big(({\mathfrak a}\cdot {\mathfrak b})^{\lambda}\big)\subseteq I_p({\mathfrak a}^{\lambda})\cdot I_p({\mathfrak b}^{\lambda}).$$
\end{proposition}

\begin{proof}
Consider the diagonal embedding $\Delta\colon X\hookrightarrow X\times X$.
If the assertion in the proposition holds for the ideals $\widetilde{\fra}$ and  $\widetilde{\frb}$ on $X\times X$ given 
by pulling back $\fra$ and $\frb$ respectively, via the first and second projections, then
 it follows from Corollary~\ref{cor_pull_back} and Proposition~\ref{smooth_pull_back}
that if ${\mathfrak c}=\widetilde{\mathfrak a}\cdot \widetilde{\mathfrak b}$, then
$$I_p\big(({\mathfrak a}\cdot {\mathfrak b})^{\lambda}\big)=I_p\big(({\mathfrak c}\cdot \shO_X)^{\lambda}\big)
\subseteq I_p({\mathfrak c}^{\lambda})\cdot\shO_X$$
$$\subseteq 
\big(I_p(\widetilde{\mathfrak a}^{\lambda})\cdot \shO_X\big)\cdot \big(I_p(\widetilde{\mathfrak b}^{\lambda})\cdot \shO_X\big)=I_p({\mathfrak a}^{\lambda})\cdot
I_p({\mathfrak b}^{\lambda}).$$

Therefore we may assume that $X=X_1\times X_2$ and that ${\mathfrak a}={\mathfrak a}_1\cdot\shO_X$ and ${\mathfrak b}={\mathfrak a}_2\cdot\shO_X$,
where ${\mathfrak a}_i$ are ideals on $X_i$. In this case, by combining Propositions~\ref{behaviour_under_inclusion}
 and \ref{smooth_pull_back}, we see that
$$I_p\big(({\mathfrak a}\cdot {\mathfrak b})^{\lambda}\big)\subseteq I_p({\mathfrak a}^{\lambda})\cap I_p({\mathfrak b}^{\lambda})=
\big(I_p({\mathfrak a}_1^{\lambda})\cdot \shO_X\big)\cap
\big(I_p({\mathfrak a}_2^{\lambda})\cdot \shO_X\big)$$
$$=I_p({\mathfrak a}_1^{\lambda})\otimes_{\mathbf C}I_p({\mathfrak a}_2^{\lambda})=I_p({\mathfrak a}^{\lambda})\cdot I_p({\mathfrak b}^{\lambda}).$$
\end{proof} 

\begin{remark}
A similar argument shows that under the assumptions of Proposition~\ref{subadditivity}, for every $\lambda,\mu\in \QQ\cap (0,1]$ and every $p\geq 0$,
we have
$$I_p(\fra^{\lambda}\frb^{\mu})\subseteq I_p(\fra^{\lambda})\cdot I_p(\frb^{\mu}).$$
\end{remark}

\medskip

We end this section with a triviality criterion for all Hodge ideals $I_p(\fra^{\lambda})$, where $\fra$ is 
any nonzero ideal on $X$.
Given a point $x\in X$, defined by the ideal ${\mathfrak m}_x$, we denote by ${\rm ord}_x(\fra)$
the largest nonnegative integer $q$ such that $\fra\subseteq {\mathfrak m}_x^q$.

\begin{proposition}\label{triviality}
If $X$ is a smooth $n$-dimensional variety, $x\in X$ is a point in the support of the subscheme defined by the ideal $\fra \subseteq \shO_X$, and 
$\lambda\in (0,1]$, then the following are equivalent:
\begin{enumerate}
\item[i)] For all $p\geq 0$, we have $I_p(\fra^{\lambda})_x=\shO_{X,x}$.
\item[ii)] There is $p\geq n$ such that $I_p(\fra^{\lambda})_x=\shO_{X,x}$.
\item[iii)] We are in one of the following two situations: either ${\rm ord}_x(\fra)=1$, or in a suitable neighborhood of $x$,
we have $\fra=\shO_X(-mZ)$ for some smooth divisor $Z$, and $2\leq m\leq \frac{1}{\lambda}$.
\end{enumerate}
\end{proposition}

\begin{proof}
We may assume that $X$ is affine,  and we let $D$ be the divisor defined by a general linear combination of some generators of $\fra$.
Given $p\geq 0$, it follows from Theorem~\ref{prop1} that $I_p(\fra^{\lambda})_x=\shO_{X,x}$ if and only if there is such a $D$ with
$I_p(\lambda D)_x=\shO_{X,x}$. In fact, in this case the same equality holds for all general $D$; this is a consequence of the Semicontinuity Theorem
for Hodge ideals (see \cite[Theorem~14.1]{MP4}). 

On the other hand, if $I_p(\lambda D)_x=\shO_{X,x}$ for some $p\geq n$, then $D_{\rm red}$ is smooth 
at $x$ (see \cite[Corollary~10.7]{MP4}). If this is the case, after replacing $X$ by a suitable neighborhood of $x$, we may assume that 
$Z=D_{\rm red}$ is smooth and $D=mZ$. If $m=1$, then we clearly have ${\rm ord}_x(\fra)=1$. On the other hand, if $m\geq 2$, then 
$D$ being general implies that $D={\rm div}(\fra)$, hence $\fra=\shO_X(-mZ)$. The inequality 
$\lambda m\leq 1$ follows from the fact that, since $Z$ is smooth, we have
$$I_p(\lambda mZ)=\shO_X\big((1-\lceil \lambda m\rceil)Z)$$
(see \cite[\S 3,4]{MP4}). This proves the implication ii)$\Rightarrow$iii). 

The implication iii)$\Rightarrow$i) follows immediately from the fact that, as we have already seen,
for a smooth divisor  $Z$ we have $I_p(\lambda Z)=\shO_X$ for all $p\geq 0$ and $\lambda\in (0,1]$. Since the implication
i)$\Rightarrow$ii) is trivial, this completes the proof of the proposition.
\end{proof}

As mentioned in the Introduction, further tools from the study of Hodge ideals of divisors are still missing, mainly due to the lack of a direct connection with Hodge theory. For example, at least at the moment, there is no $\Dmod_X$-module (of Hodge theoretic origin) associated naturally to the ideals $I_p (\fra)$.
A natural question is the following:

\begin{question}\label{vanishing}
Is there a vanishing theorem for Hodge ideals associated to ideals? More precisely, assuming that $X$ is a smooth projective variety, 
$\fra$ is a nonzero ideal on $X$ and $A$ is a line bundle on $X$, what are the conditions 
$\fra$, $A$ and $p$ must satisfy in order to have 
$$H^i \big(X, \omega_X \otimes A \otimes I_p (\fra)\big) = 0, \,\,\,\,\,\, {\rm for~all} \,\,\,\,i > 0.$$
\end{question}

Here one is looking for a statement in analogy with the vanishing theorem for Hodge ideals of divisors, see \cite[Theorem~F]{MP1} and 
\cite[Theorem~12.1]{MP3}, and with that for multiplier ideals associated to ideals, see \cite[Corollary 9.4.15]{Lazarsfeld}.

\section{Examples}

In this section we provide a few concrete calculations of Hodge ideals associated to ideals; note that even in 
the case of powers of the maximal ideal this is quite involved. We also give some examples of pathological 
behavior of higher Hodge ideals, compared to the case of multiplier ideals.

First, in light of the Proposition \ref{rem1}, we see that if $X$ is affine and $h$  is a general linear combination of a system of generators of ${\mathfrak a}$, then
$$I_0(h^{\lambda})=I_0({\mathfrak a}^{\lambda})\quad \text{for all}\quad\lambda\in (0,1]\cap\QQ.$$
We give two examples showing that the corresponding assertion can fail for $p>0$, even when $\lambda=1$.

\begin{example}
Let ${\mathfrak a}=(xy,xz)\subseteq \CC[x,y,z]$. Note that for every $(a,b)\in\CC^2\smallsetminus\{(0,0)\}$, 
the divisor $D_{a,b}$ in $\AAA^3$ defined by $axy+bxz$ is reduced, with simple normal crossings, and so
by \cite[Proposition~8.2]{MP1} we have
$$I_1(D_{a,b})=(x,ay+bz).$$
We thus see that $I_1({\mathfrak a})=(x,y,z)$, but $I_1(D_{a,b})\neq (x,y,z)$ for any $(a,b)\neq (0,0)$.
\end{example}

\begin{example}
Let ${\mathfrak a}=(x^2,y^3)\subseteq \CC[x,y]$. If $D_{a,b}$ is the divisor in $\AAA^2$ defined by $h=ax^2+by^3$, with $a,b\neq 0$, an easy computation 
based on
\cite[Corollary~17.8]{MP1} gives
$$I_2(D_{a,b})=(x^3,x^2y^2,xy^3,3ax^2y-by^4).$$
We deduce from Theorem \ref{prop1}i) that
$$I_2({\mathfrak a})=(x^3,x^2y,xy^3, y^4)\neq I_2(D_{a,b})\quad \text{for all}\quad a,b\neq 0.$$
\end{example}

We now give an example in which we can compute the Hodge ideal of an ideal, while we do not have a closed formula for the corresponding Hodge ideal of a general member of the ideal.

\begin{example}\label{powers_maximal}
We compute the Hodge ideals associate to powers of maximal ideals. Let ${\mathfrak m}_x$ be the ideal defining the point $x$ on a smooth variety $X$ of dimension $n\geq 2$. We will show that if $N\geq 1$ and $\mu(N,p,n)=(p+1)(N-1)-n+\lceil n/N\rceil$, then
\begin{equation}\label{eq_formula_power}
I_p({\mathfrak m}_x^N) = \left\{
\begin{array}{cl}
\shO_X, & \text{if}\quad p+1\leq\frac{n}{N}; \\[2mm]
{\mathfrak m}_x^{\mu(N,p,n)}, &  \text{if} \quad p+1>\frac{n}{N}.
\end{array}\right.
\end{equation}
Note that if $p+1\geq \frac{n}{N}$, then $\mu(N,p,n)\geq 0$. 

For $N=1$, the above formula says that $I_p({\mathfrak m}_x)=\shO_X$ for all $p$, which is clear (see Proposition~\ref{triviality}). From now on we assume $N\geq 2$.
By taking an \'{e}tale map $U\to\AAA^n$ that maps $x$ to $0$, where $U$ is an open neighborhood of $x$, using 
Proposition~\ref{smooth_pull_back} we may assume that 
$X=\AAA^n$ and ${\mathfrak m}_x=(x_1,\ldots,x_n)$. In this case, since ${\mathfrak m}_x^N$ is preserved by all linear changes of variables, every $I_p({\mathfrak m}_x^N)$ has the same property, hence it is a power of 
${\mathfrak m}_x$. It follows that given a system of homogeneous generators
of $I_p({\mathfrak m}_x^N)$, we only need to determine the minimal degree of these generators.

Let $D$ be the divisor in $\AAA^n$ defined by a general linear combination $f$ of the monomials of degree $N$. In particular $f$ is a homogeneous polynomial, with an isolated singularity at $0$. Note that $I_p(D)$ is a homogeneous ideal, but might not be monomial.
We need to show that if $\nu(N,p,n)$ is the minimal degree of a homogeneous element of $I_p(D)$, then
$\nu(N,p,n)=\mu(N,p,n)$ if $p+1>\frac{n}{N}$ and $\nu(N,p,n)=0$, otherwise.

The key ingredient is an inductive formula for computing the Hodge ideals of such a polynomial $f$;  according to \cite[Corollary~B]{Zhang}, inspired in turn by a result in \cite{Saito-HF}, for 
every $p \ge 1$ we have 
\begin{equation}\label{homogeneous_formula}
I_p (D)  = \sum_{{\rm deg} (v_j) \ge (p +1)N - n} \shO_X \cdot v_j + \sum_{1\le i \le n, ~g \in I_{p-1}(D)} \shO_X \cdot 
(f \partial_ig - p g \partial_if),
\end{equation}
where the first sum is taken over those $v_j$ in a basis of monomials for the Milnor algebra 
$$S=\CC [X_1, \ldots, X_n] / (\partial_1f, \ldots, \partial_nf),$$
whose degree is at least $(p +1)N - n$.

We prove the formula for $\nu(N,p,n)$ by induction on $p$, the case $p=0$ being clear, 
by Proposition~\ref{rem1}
and the well-known formula for $\I({\mathfrak m}_x^{\lambda})$ (see \cite[Example~9.2.14]{Lazarsfeld}): 
$$I_0(D)=\I\big((1-\epsilon)D\big)=\I({\mathfrak m}_x^{N(1-\epsilon)})={\mathfrak m}_x^{N-n},\quad\text{where}\quad 0<\epsilon\ll 1,$$
with the convention that the last term is $\shO_X$ when $N<n$.

 If $p+1\leq\frac{n}{N}$, then
$I_p(D)=\shO_X$ by (\ref{homogeneous_formula}) since $1$ is part of a monomial basis of the Milnor algebra (recall that we assume $N\geq 2$) 
of degree $0\geq (p+1)N-n$. 
Suppose now that $p$ is positive, with $p+1>\frac{n}{N}$. Note that if $g$ is a homogeneous polynomial of degree $q$ in $I_{p-1}(D)$, then by (\ref{homogeneous_formula})
 all $f \partial_ig - p g \partial_if$ lie in $I_p(D)$; if nonzero, these are homogeneous of degree $N+q-1$. If $q=\nu(N,p-1,n)$, then not all
these can be $0$: otherwise we have $\partial_i(g/f^p)=0$ for all $i$, hence $g/f^p$ is a constant, and thus $q=pN$; however, using the formula
for $\nu(N,p-1,n)$ given by the induction hypothesis, we see that $\nu(N,p-1,n)<pN$. 

We also note that we get a contribution to $I_p(D)$ from the first sum in (\ref{homogeneous_formula}) if and only if $(p+1)N-n\leq n(N-2)$,
and in this case the contribution consists of monomials of degree $\geq (p+1)N-n$,
with equality for some monomials.
 Indeed, since $\partial_1f,\ldots,\partial_nf$ form a regular sequence
of homogeneous forms of degree $N-1$, the Hilbert series of $S$ is given by 
$$\frac{(1-t^{N-1})^n}{(1-t)^n}=(1+t+\cdots+t^{N-2})^n,$$
hence for a nonnegative integer $d$ we have $S_d\neq 0$ if and only if $d\leq n(N-2)$.
By combining these observations, we conclude from (\ref{homogeneous_formula}) that 
\begin{equation}\label{formula_nu}
\nu(N,p,n)= \left\{
\begin{array}{cl}
\min\{\nu(N,p-1,n)+N-1, (p+1)N-n\}, & \text{if}\quad (p+1)N\leq n(N-1); \\[2mm]
\nu(N,p-1,n)+N-1\, &  \text{if} \quad (p+1)N>n(N-1).
\end{array}\right.
\end{equation}

We distinguish two cases. If $p>\frac{n}{N}$, then we see using the induction hypothesis and an easy computation that
$$\nu(N,p-1,n)+N-1=\mu(N,p-1,n)+N-1\leq (p+1)N-n,$$
hence we deduce using (\ref{formula_nu}) that $\nu(N,p,n)=\mu(N,p-1,n)+N-1=\mu(N,p,n)$.

Suppose now that $p\leq\frac{n}{N}$, hence by the induction hypothesis we have $\nu(N,p-1,n)=0$. We further distinguish two possibilities.
If $pN\in\{n-1,n\}$, then we again have $N-1\leq (p+1)N-n$, hence $\nu(N,p,n)=N-1$ by (\ref{formula_nu}). Moreover, in this case it is easy to see
that $\mu(p,N,n)=N-1$, hence we are done.

On the other hand, if $pN\leq n-2$, then $(p+1)N\leq n(N-1)$ (we use the fact that $N\geq 2$) and 
$(p+1)N-n\leq N-1$, so that it follows from (\ref{formula_nu}) that $\nu(N,p,n)=(p+1)N-n$.
Note also that in this case we have $\lceil n/N\rceil=p+1$, hence
$\mu(N,p,n)=(p+1)N-n$. This completes the proof of (\ref{eq_formula_power}).
\end{example}

\begin{example}\label{example_diagonal}
Let $\fra=(x_1^N,\ldots,x_n^N)\subseteq \CC[x_1,\ldots,x_n]$, with $n,N\geq 2$.
We show that if ${\mathfrak m}=(x_1,\ldots,x_n)$, then
\begin{equation}\label{formula_mu}
I_1(\fra)= \left\{
\begin{array}{cl}
\CC[x_1,\ldots,x_n], & \text{if}\quad N\leq \frac{n}{2}; \\[2mm]
(x_1^{N-1},\ldots,x_n^{N-1})+{\mathfrak m}^{2N-n}, & \text{if}\quad \frac{n}{2}\leq N\leq n; \\[2mm]
(x_1^{N-1},\ldots,X_n^{N-1})\cdot{\mathfrak m}^{N-n}+{\mathfrak m}^{2N-n}, &  \text{if} \quad N\geq n.
\end{array}\right.
\end{equation}
Suppose that $N>n$. Let $D$ be the divisor defined by a general linear combination
$$f=\sum_{i=1}^n\alpha_ix_i^N.$$
Again, $f$ is homogeneous of degree $N$, having an isolated singularity at $0$, hence we can use
the formula (\ref{homogeneous_formula}). 

In this case the Milnor algebra is given by
$$S=\CC[x_1,\ldots,x_n]/(x_1^{N-1},\ldots,x_n^{N-1}),$$
hence the contribution of the first sum in (\ref{homogeneous_formula}) to $I_1(D)$ consists of
$$(x_1^{a_1}\cdots x_n^{a_n}\mid a_i\leq N-2\,\,\text{for all}\,\,i,~ a_1+\cdots+a_n\geq 2N-n).$$

Note that since $\fra$ is a monomial ideal, it is preserved by the standard action of $(\CC^*)^n$ on $\AAA^n$,
hence the same holds for $I_1(\fra)$. Therefore $I_1(\fra)$ is a monomial ideal as well. It follows that $I_1(\fra)$
is generated by the monomials that appear with nonzero coefficient in the polynomials in $I_1(D)$, for $D$ as above. 

Since ${\mathfrak m}^N$ is the integral closure of $\fra$ and since multiplier ideals do not change after replacing an ideal by its integral closure (see \cite[Corollary~9.6.17]{Lazarsfeld}), we see as in Example~\ref{powers_maximal} that
$I_0(D)={\mathfrak m}^{N-n}$. Thus the contribution of the second sum in (\ref{homogeneous_formula})
to $I_1(D)$ consists of the ideal generated by
$f\partial_ig-Nx_i^{N-1}g$, where $g$ varies over the monomials in ${\mathfrak m}^{N-n}$ and $1\leq i\leq n$. Since the coefficients of $f$ 
are general, it is clear that the monomials that appear in $f\partial_ig-Nx_i^{N-1}g$ are $x_i^{N-1}g$ and $x_j^N\partial_ig$, with $1\leq j\leq n$.
The ideal generated by these monomials is
$(x_1^{N-1},\ldots,x_n^{N-1})\cdot {\mathfrak m}^{N-n}$.

By combining the two contributions, we conclude that 
$$I_1(\fra)=(x_1^{N-1},\ldots,x_n^{N-1})\cdot {\mathfrak m}^{N-n}+{\mathfrak m}^{2N-n},$$
which proves our formula for $N>n$. The proofs in the other two cases are similar, but easier.
\end{example}

\begin{example}[Non-invariance under integral closure]
Recall that if $\fra$ and $\frb$ are two nonzero ideals on $X$, with the same integral closure, then
$$\I(\fra^{\lambda})=\I(\frb^{\lambda})\quad\text{for all}\quad\lambda>0$$
(see \cite[Corollary~9.6.17]{Lazarsfeld}). 
This property fails for Hodge ideals: consider, for example, $\fra=(x^N,y^N,z^N)$ and $\frb=(x,y,z)^N$ in $\CC[x,y,z]$, for $N\geq 3$.
Note that $\frb$ is the integral closure of $\fra$, while it follows from Examples~\ref{powers_maximal} and \ref{example_diagonal}
that $I_1(\fra)$ is strictly contained in $I_1(\frb)$. 
\end{example}

\begin{example}[Failure of the asymptotic property]\label{no_asympt}
For multiplier ideals, it follows immediately from their definition that 
$$\I \big((\fra^\ell)^\frac{\lambda}{\ell} \big) = \I \big((\fra^{k\ell})^\frac{\lambda}{k\ell}\big)$$
for all integers $k, \ell > 0$. The inclusion ``$\subseteq$" is crucial 
 for the construction of \emph{asymptotic} multiplier ideals; 
 see \cite[\S11.1]{Lazarsfeld}. This inclusion might not hold for higher Hodge ideals. Consider for instance the maximal ideal $\fra = (x_1, \ldots, x_n) \subset \CC[x_1, \ldots, x_n]$, with $n \ge 3$, and fix an integer $m > 0$. Let $D$ be the zero locus of a general linear combination of monomials of degree 
$m$ in the $x_i$, so that $D$ has an isolated ordinary singularity of multiplicity $m$ at the origin.
An easy application of \cite[Corollary~B]{Zhang} (see also 
\cite[Example~11.7]{MP4}) gives 
$I_1 \big(\frac{1}{m} D\big) = \frak{m}_x^{m + 1 - n}$ for $m\geq n-1$; in this case Theorem \ref{prop1}i) implies
$I_1 \big(\frac{1}{m} \fra^m\big) = \frak{m}_x^{m + 1 - n}$. This means that for $\ell > n -1$ and $k > 1$ 
we have in fact the strict inclusion
$$I_1 \big( (\fra^{k\ell})^\frac{1}{k\ell} \big) \subsetneq I_1 \big( (\fra^{\ell})^\frac{1}{\ell} \big).$$
It is an interesting question if, or when, some type of asymptotic construction can be performed in this context.
\end{example}

\begin{example}
For effective divisors $D$ and $E$ on a smooth variety $X$, with $D+E$ reduced, it is shown in \cite[Theorem~B]{MP2} that we have
$$I_p(D+E)\subseteq\sum_{i+j=p}I_i(D)\cdot I_j(E)\cdot\shO_X(-jD-iE)\quad\text{for all}\quad p\geq 0.$$
We used this for instance to deduce the inclusion in Proposition~\ref{subadditivity} in the case of locally principal ideals and $\lambda=1$.
One could ask whether for arbitrary nonzero ideals ${\mathfrak a}$ and ${\mathfrak b}$ such that 
${\rm div}(\fra)$ and ${\rm div}(\fra)$ have no common components,
we have
\begin{equation}\label{eq1_rem_subadditivity}
I_p({\mathfrak a}\cdot {\mathfrak b})\subseteq\sum_{i+j=p}I_i({\mathfrak a})\cdot I_j({\mathfrak b})\cdot {\mathfrak a}^{j}\cdot{\mathfrak b}^i\quad\text{for all}\quad p\geq 0.
\end{equation}
It is easy to deduce that this still holds if either ${\mathfrak a}$ or ${\mathfrak b}$ is locally principal. 
However, it does not hold in general. Suppose, for example, that $X=\AAA^{2n}$ with coordinates $x_1,\ldots,x_n,y_1,\ldots,y_n$,  while
${\mathfrak a}=(x_1,\ldots,x_n)$ and ${\mathfrak b}=(y_1,\ldots,y_n)$. Note that  $I_i({\mathfrak a})=I_i({\mathfrak b})=\shO_X$ (see Proposition~\ref{triviality}),
hence (\ref{eq1_rem_subadditivity}) says in this case that
$$I_p({\mathfrak a}\cdot {\mathfrak b})\subseteq\sum_{i+j=p}(x_1,\ldots,x_n)^j\cdot (y_1,\ldots,y_n)^i=(x_1,\ldots,x_n,y_1,\ldots,y_n)^p.$$
However, it follows from \cite[Corollary~D]{MP2} that if $f=\sum_{i=1}^nx_iy_i$, then 
$I_k(f)=\shO_X$ for $p\leq n-1$. Therefore (\ref{eq1_rem_subadditivity}) fails for $n\geq 2$. 
\end{example}

\section{Generic minimal exponent}

In this section we define and study an extension of the concept of minimal exponent of a hypersurface \cite{Saito-B}, \cite{Saito-MLCT} (see also \cite{MP3}, \cite{MP5} for a recent study and applications) to the case of arbitrary subschemes. As always, we work on a smooth variety $X$ of dimension $n$.

Recall first that an important invariant of the singularities of a nonzero $f\in\shO_X(X)$ is the \emph{Bernstein-Sato polynomial}
$b_f(s)\in {\mathbf C}[s]$ of $f$. The roots of $b_f$ are negative rational numbers by a theorem of Kashiwara 
\cite{Kashiwara}. From now on we assume that $f$ is not invertible, in which case $b_f(-1)=0$. The negative of the greatest root of $b_f(s)/(s+1)$ is the \emph{minimal exponent} $\widetilde{\alpha}(f)$ of $f$ (with the convention that if $b_f(s)=s+1$, which is the case if and only if $f$ defines a smooth hypersurface, then $\widetilde{\alpha}(f)=\infty$). By a result of Lichtin and Koll\'{a}r (see \cite[Theorem~10.6]{Kollar}), the negative of the greatest root of $b_f(s)$ is the \emph{log canonical threshold} ${\rm lct}(f)$, hence ${\rm lct}(f)=\min\{1,\widetilde{\alpha}(f)\}$. For an introduction to the log canonical threshold and its relation to multiplier ideals,
we refer to \cite[Chapter~9]{Lazarsfeld}. 

We will be mostly using a local version of the minimal exponent: given $x$ in the zero-locus of $f$, if $U$ is an open neighborhood of $x$,
then $\widetilde{\alpha}(f\vert_U)\geq \widetilde{\alpha}(f)$. Moreover, if $U$ is small enough, then $\widetilde{\alpha}(f\vert_U)$ is independent of $U$;
the common value is the \emph{minimal exponent} $\widetilde{\alpha}_x(f)$ of $f$ at $x$. 

\begin{remark}
The global and local minimal exponents of $f$ were denoted in \cite{MP3} by $\widetilde{\alpha}_{f}$ and $\widetilde{\alpha}_{f,x}$, respectively,
in line with the notation from \cite{Saito-B}, \cite{Saito-MLCT}. 
However, for what follows below we found the present notation more convenient. 
\end{remark}

The minimal exponent is related to Hodge ideals as follows: if $f$ defines a divisor $D$ which is reduced in a neighborhood of $x$, then
\begin{equation}\label{char_min_exp}
I_p (\lambda D)_x = \shO_{X,x} \iff p + \lambda \le \widetilde{\alpha}_x(f)
\end{equation}
(see \cite[Corollary~C]{MP3}). Note that from the point of view of the minimal exponent, the interesting case is that when $D$ is reduced in some
neighborhood of $x$; otherwise ${\rm lct}_x(f)<1$ and $\widetilde{\alpha}_x(f)={\rm lct}_x(f)$.

We will make use of the following semicontinuity property of minimal exponents
for hypersurfaces. Suppose that we have a smooth morphism of complex algebraic varieties $\pi\colon W\to T$, with a section $s\colon T\to W$. 
Given $f\in\shO_W(W)$ such that the restriction $f_t$ to the fiber $\pi^{-1}(t)$ is nonzero for every $t\in T$, the
function 
$$T\ni t\to \widetilde{\alpha}_{s(t)}(f_t)\in {\mathbf R}_{>0}\cup\{\infty\}$$
is lower semicontinuous (see \cite[Theorem E(2)]{MP3}). In fact, the proof in \emph{loc. cit.} shows something stronger:
for every $\alpha>0$, the set $\{t\in T\mid \widetilde{\alpha}_{s(t)}(f_t)\geq\alpha\}$ is open in $T$.
Since a countable intersection of nonempty open subsets of $T$ is nonempty, it follows that 
the set $\{\widetilde{\alpha}_{s(t)}(f_t)\mid t\in T\}$ has a maximum, which is achieved on an open subset of $T$. Arguing by Noetherian induction,
we deduce that this set is in fact finite.

We now turn to the case of ideals. Consider a nonzero
ideal $\fra \subseteq \shO_X$ and a point $x$ in the zero-locus of $\fra$; since we are interested in a local study around $x$, we assume that $X$ is affine, and $\fra$ is generated by $f_1, \ldots, f_r$ in $\shO_X(X)$.  

\begin{definition}
The  \emph{generic minimal exponent} of $\fra$ at $x$ is defined as 
$$\overline{\alpha}_x(\fra): = \widetilde{\alpha}_x(f),$$
where $f = \sum_{i =1}^r \lambda_i f_i$ is a general linear combination of the generators of $\fra$. 
\end{definition}

\begin{remark}
The fact that for a general combination $f$ as above the value of $\widetilde{\alpha}_x(f)$
is constant follows from the above discussion about the semicontinuity of the minimal exponent. Furthermore,
it is straightforward to see that this value is independent of the choice of generators of $\fra$. 
\end{remark}

\begin{remark}\label{local-coh-comment}
A priori it would make sense to simply call $\overline{\alpha}_x(\fra)$ the \emph{minimal exponent} of $\fra$ and denote it by $\widetilde{\alpha}_x(\fra)$,
extending the terminology and notation from the case of hypersurfaces. However, we prefer to keep these for a different invariant, defined in terms of
the Bernstein-Sato polynomial $b_{\fra,x}(s)$ in the sense of \cite{BMS}. If $\fra$ defines a closed subscheme $Z$ of codimension $r$ at $x$,
reduced in some neighborhood of $x$, then one can deduce from \cite[Theorem~2]{BMS} that $b_{\fra,x}(-r)=0$; we define the \emph{minimal exponent} 
$\widetilde{\alpha}_x(\fra)$ as the negative of the largest root of $b_{\fra,x}(s)/(s+r)$. This is in general different from $\overline{\alpha}_x(\fra)$, and seems to be related more naturally to the Hodge filtration on local cohomology. We hope to study this relationship in future work. 
 \end{remark}

\begin{proposition}\label{min_exp_equiv}
If $\fra$ is not radical in codimension $1$ around $x$, then
$\overline{\alpha}_{x}(\fra)$ is equal to the log canonical threshold ${\rm lct}_x (\fra)$
of $\fra$ at $x$. On the other hand, if 
$\fra$ is radical in codimension $1$ around $x$, then
\begin{equation}\label{eq2_min_exp_equiv}
I_p (\fra^\lambda)_x = \shO_{X,x} \iff p + \lambda \le \overline{\alpha}_x(\fra).
\end{equation}
\end{proposition}

\begin{proof}
If $\fra$ is not radical in codimension $1$ around $x$ and $f$ is a general linear combination of generators of $\fra$, 
 then $f$  defines a divisor having a non-reduced component containing $x$. 
We therefore have ${\rm lct}_x(f) <1$, and thus 
$${\rm lct}_x (\fra) = {\rm lct}_x(f) = \widetilde{\alpha}_x(f),$$
where the first equality follows from \cite[Proposition~9.2.28]{Lazarsfeld} and the description of the log canonical threshold 
via multiplier ideals.  

Suppose now that $\fra$ is reduced in codimension $1$ around $x$.
If $\lambda>0$ is a rational number and 
$f$ is a general linear combination of generators of $\fra$,
defining a divisor $D$ which is reduced in some neighborhood of $x$,
then $\overline{\alpha}_x(\fra)=\widetilde{\alpha}_x(f)$. Moreover, we have 
$I_p (\fra^\lambda)_x = \shO_{X,x}$ if and only if $I_p(\lambda D)_x=\shO_{X,x}$
(for the ``only if" part, we use that $\shO_{X,x}$ is a local ring). The equivalence in 
(\ref{eq2_min_exp_equiv}) then follows from (\ref{char_min_exp}).
\end{proof}

\begin{remark}\label{rmk_min_exp_equiv}
If $p=0$, then the equivalence in (\ref{eq2_min_exp_equiv}) also holds when $\fra$ is not radical in codimension 1 around $x$.
Indeed, this follows from the description of $I_0(\fra^{\lambda})$ as a multiplier ideal in Proposition~\ref{rem1}
and the characterization of ${\rm lct}_x(\fra)$ via multiplier ideals.
\end{remark}

\begin{example}\label{eg1}
We collect a first few examples here. The case of general monomial ideals is discussed in Example \ref{monomial} below.

\noindent
(1) We have $\overline{\alpha}_x(\fra) = \infty$ if and only if ${\rm ord}_x (\fra) =1$, meaning $\fra 
\not\subseteq \frak{m}_x^2$.

\noindent
(2) If $N\ge 2$, then $\overline{\alpha}_x({\mathfrak m}_x^N) = \frac{n}{N}$, since the same is true for a 
hypersurface having multiplicity $N$ at $x$ and whose projectivized tangent cone at $x$  is smooth; 
see \cite[(4.1.5)]{Saito-HF} (cf. also \cite[Theorem~E(3)]{MP3}).

\noindent
(3) In general, if ${\rm ord}_x (\fra) =N \ge 2$, then $\overline{\alpha}_x(\fra) \le \frac{n}{N}$. This follows using 
\cite[Theorem~E(3)]{MP3}.
\end{example}

As in the case of hypersurfaces, we have:

\begin{proposition}\label{std_ineq}
For every ideal $\fra$, we have 
$$\overline{\alpha}_x(\fra) \ge {\rm lct}_x (\fra).$$
Moreover, this is an equality if ${\rm lct}_x (\fra) < 1$.  
\end{proposition}

\begin{proof}
It is shown in \cite[Proposition~2.1]{CM} that if $f$ is a general linear combination of generators of $\fra$,
then $\widetilde{\alpha}_x(f)\geq{\rm lct}_x(\fra)$. 
The argument uses \cite[Corollary~D]{MP3}, which gives a lower bound for $\widetilde{\alpha}_x(f)$ 
in terms of discrepancies on a log resolution. This implies the first assertion.
Another proof follows from Proposition \ref{min_exp_coeff} below; see Remark \ref{new_ineq}.
The second assertion follows as in the proof of Proposition~\ref{min_exp_equiv}.
\end{proof}

\begin{proposition}\label{behavior_inclusion}
If $\fra\subseteq\frb$ are nonzero ideals on $X$ and $x$ lies in the zero-locus of $\frb$, then
$$\overline{\alpha}_x(\fra)\leq\overline{\alpha}_x(\frb).$$
\end{proposition}

\begin{proof}
Let $f$ be a general linear combination of generators of $\fra$ and $g$ a general linear combination
of generators of $\frb$, so that
$$\overline{\alpha}_x(\fra)=\widetilde{\alpha}_x(f)\quad\text{and}\quad \overline{\alpha}_x(\frb)=\widetilde{\alpha}_x(g).$$
Since $f\in\frb$, it follows from the semicontinuity property of the minimal exponents for hypersurfaces that 
$\widetilde{\alpha}_x(f)\leq\widetilde{\alpha}_x(g)$, which gives the assertion in the proposition.
\end{proof}

The following series of properties of the minimal exponent of an ideal follows without much effort from the analogous properties proved in the case of divisors in 
\cite[Theorem~E and \S6]{MP3}.

\begin{proposition}\label{properties_minimal_exponent}
(1) For every smooth subvariety $Y \subseteq X$, every ideal $\fra$ on $X$ such that $\fra\cdot \shO_Y\neq 0$, 
and every $x$ in the zero-locus of $\fra\cdot\shO_Y$, we have 
$$\overline{\alpha}_x(\fra\cdot \shO_Y) \leq \overline{\alpha}_x(\fra).$$

\noindent
(2) For every ideal $\fra$ and every $\alpha > 0$, the set 
$$\{ x \in V(\fra)~|~\overline{\alpha}_x(\fra) \geq\alpha\}$$
is open in $X$.

\noindent
(3) More generally, let $f \colon X \to T$ be a smooth morphism and $s \colon T \to X$ a section of $f$. If
$\fra$ is a nonzero ideal on $X$ that vanishes on $s(T)$ and such that $\fra\cdot\shO_{X_t}$ is not zero
for any fiber $X_t$ of $f$ over $t\in T$, then for every $\alpha>0$, the set
$$\{t\in T~|~  \overline{\alpha}_{s(t)}(\fra\cdot \shO_{X_t})\geq\alpha\}$$
is open in $T$.

\noindent
(4) If $\fra$ and $\frb$ are nonzero ideals vanishing at $x \in X$, then 
$$\overline{\alpha}_x(\fra + \frb) \le \widetilde{\alpha}_x(\fra) + \widetilde{\alpha}_x(\frb).$$
\end{proposition}

\begin{example}\label{monomial}
We show that if $\fra$ is a monomial ideal in ${\mathbf C}[x_1,\ldots,x_n]$, with ${\rm ord}_0(\fra)>1$, then 
$\overline{\alpha}_0(\fra) = {\rm lct}_0 (\fra)$. Recall that in this case, by a result of Howald \cite{Howald} we have 
${\rm lct}_0(\fra)=1/c$, where if $P_{\fra}$ is the Newton polyhedron of $\fra$ (that is, $P_{\fra}$ is the convex hull of $u+{\mathbf R}_{\geq 0}^n$, for the monomials 
$x^u\in \fra$), we have $c=\min\{t>0\mid (t,\ldots,t)\in P_{\fra}\}$. 

Note now that if ${\mathfrak m}=(x_1,\ldots,x_n)$, then 
$$0\leq \overline{\alpha}_0(\fra+{\mathfrak m}^N)-\overline{\alpha}_0(\fra)\leq \frac{n}{N}.$$
Indeed, the first inequality follows from Proposition~\ref{behavior_inclusion}, while the second follows from 
Proposition~\ref{properties_minimal_exponent}(4)
and Example~\ref{eg1}(2).
We similarly have
$$0\leq {\rm lct}_0(\fra+{\mathfrak m}^N)-{\rm lct}_0(\fra)\leq\frac{n}{N}$$
(see \cite[Corollary~9.5.28]{Lazarsfeld}).
By letting $N$ go to infinity, we see that it is enough to show that $\overline{\alpha}_0(\fra)= {\rm lct}_0(\fra)$
when $\fra$ is a monomial ideal defining a scheme supported at $0$ and such that ${\rm ord}_0(\fra)>1$. If $f$
is a general linear combination of monomial generators of $\fra$, then the hypersurface defined by $f$ has an isolated singular 
point at $0$. Moreover, it is nondegenerate with respect to its Newton polyhedron, in which case it is well-known that $\widetilde{\alpha}_0(f)
=1/c$ (see \cite{Varchenko}, \cite{EhlersLo}, or \cite{Saito-exponents}).
\end{example}

We can define a global version of the generic minimal exponent, as follows.
For any proper nonzero ideal $\fra$ on $X$, we put 
\begin{equation}\label{def_global}
\overline{\alpha}(\fra) : = \underset{x \in V (\fra)}{\rm min} \overline{\alpha}_x(\fra).
\end{equation}
Note that since we work over ${\mathbf C}$, a countable intersection of Zariski open subsets of 
an irreducible algebraic variety has nonempty intersection. Using this, it follows easily 
from Proposition~\ref{properties_minimal_exponent}(2) that the set
$\{\overline{\alpha}_x(\fra)\mid x\in V(\fra)\}$ is a finite set. In particular, the minimum in (\ref{def_global})
makes sense and the set of those $x\in V(\fra)$ for which the minimum is achieved is a closed subset of $V(\fra)$. 
We also see that for every $x\in V(\fra)$, we have
$$\overline{\alpha}_x(\fra)=\max_{U\ni x}\overline{\alpha}(\fra\cdot\shO_U),$$
where the maximum is over the open neighborhoods of $x$.

\medskip

Another useful description of $\overline{\alpha}_x(\fra)$ in terms of minimal exponents of hypersurfaces is facilitated by Theorem \ref{prop1}. 
Suppose that $\fra$ is generated by $f_1, \ldots, f_r \in \shO_X(X)$ and consider in $X\times \AAA^r$ the hypersurface given by the function 
$g = \sum_{i=1}^r y_i f_i$, where $y_1, \ldots, y_r$ are the coordinates on $\AAA^r$.

\begin{proposition}\label{min_exp_coeff}
Given $x\in V(\fra)$, for ${\lambda} = (\lambda_1, \ldots, \lambda_r) \in \AAA^r$  general, 
we have 
$$\overline{\alpha}_x(\fra)= \widetilde{\alpha}_{(x,{\lambda})}(g).$$
\end{proposition}

\begin{proof}
If $\lambda$ is such that $f_{\lambda}=\sum_{i=1}^r\lambda_if_i$ is nonzero, then
$$\widetilde{\alpha}_{(x,\lambda)}(g)\geq \widetilde{\alpha}_x(f_{\lambda}).$$ 
This follows from the behavior of minimal exponents under restriction (in this case to a fiber of the projection $X \times \AAA^r \to \AAA^r$) described in \cite[Theorem~E(1)]{MP3}. We thus deduce from the definition of $\overline{\alpha}_x(\fra)$ that for $\lambda$ general, we have
$$\widetilde{\alpha}_{(x,\lambda)}(g)\geq \overline{\alpha}_x(\fra).$$ 

We next show that the opposite inequality holds for \emph{every} $\lambda \in \AAA^r$. 
If ${\rm ord}_{(x,\lambda)}(g)=1$, then ${\rm ord}_x(\fra)=1$, and the inequality holds since both sides are infinite.
Suppose now that ${\rm ord}_{(x,\lambda)}(g)\geq 2$ and consider first the case when $\fra$ is radical in codimension 1
in a neighborhood of $x$ (in which case the divisor defined by $g$ is reduced in a neighborhood of 
$\{x\}\times\AAA^r$).
Let's write
$$\widetilde{\alpha}_{(x, \lambda)}(g) = p + \alpha,$$
with $p$ an integer and $\alpha \in (0, 1]$. We deduce from 
the description of the minimal exponent of $g$ in terms of Hodge ideals that 
$$I_p (g^\alpha)_{(x,{\lambda})} = \shO_{X\times \AAA^r, (x,{\lambda})}.$$
By Proposition~\ref{min_exp_equiv}, it is enough to show that $I_p (\fra^\alpha)$ is trivial at $x$ as well. However, by Theorem \ref{prop1}(ii) we know that 
$$I_p (\fra^\alpha) = {\rm Coeff} \big(I_p (g^\alpha)\big),$$
so the result follows from the general (and easy to check) fact that if $I \subset \shO_X[y_1, \ldots, y_r]$ is an ideal which is not contained in the  maximal 
ideal $\frak{m}_{(x, {\lambda})}$, then ${\rm Coeff}(I)$ is not contained in $\frak{m}_x$.

If $\fra$ is not radical in codimension 1 around $x$, then the divisor defined by $g$ is not reduced around $(x,\lambda)$ and we
 have
 $$\overline{\alpha}_x(\fra)={\rm lct}_x(\fra)\quad\text{and}\quad \widetilde{\alpha}_{(x,\lambda)}(g)={\rm lct}_{(x,\lambda)}(g)$$
 by Proposition~\ref{min_exp_equiv}. We then argue as above, with $p=0$, using Remark~\ref{rmk_min_exp_equiv}.
\end{proof}

\begin{remark}\label{new_ineq}
The above result leads to another proof of Proposition \ref{std_ineq}. Indeed, after possibly restricting to a neighborhood of $x$, we may assume 
that ${\rm lct}_x (\fra) = {\rm lct} (\fra)$. Now by \cite[Corollary~1.2]{Mustata} we know that $\widetilde{\alpha}(g) = {\rm lct}(\fra)$. On the other hand, 
Proposition \ref{min_exp_coeff} says that for $\lambda\in\AAA^r$ general, we have
$$\overline{\alpha}_x(\fra) =  \widetilde{\alpha}_{(x, {\lambda})}(g) \ge \widetilde{\alpha}(g).$$
See also Theorem \ref{root} below and its proof for more general statements.
\end{remark}

Recall that for any nonzero ideal $\fra$ in $X$, a Bernstein-Sato polynomial $b_{\fra}(s)$ was defined in \cite{BMS},
extending the classical invariant associated to a hypersurface. For every $x\in V(\fra)$, we have a local version $b_{\fra,x}(s)$.
By Theorem 2 in \emph{loc. cit.}  the greatest root of $b_{\fra, x}(s)$ is again 
$-{\rm lct}_x (\fra)$, as in the case of hypersurfaces.
We conclude by showing that the generic minimal exponent continues to be a root as well.

\begin{theorem}\label{root}
For every $x\in V(\fra)$,  the negative of  $\overline{\alpha}_x(\fra)$ is a root of the Bernstein-Sato polynomial $b_{\fra, x}(s)$.
\end{theorem}

\begin{proof}
This is now a simple consequence of results obtained above and in \cite{Mustata}. Using the notation and statement of Proposition \ref{min_exp_coeff}, we  have 
$$\overline{\alpha}_x(\fra) = \widetilde{\alpha}_{(x, {\lambda})}(g),$$
where ${\lambda} = (\lambda_1, \ldots, \lambda_r) \in \AAA^r$ is general. By the definition of the minimal exponent of $g$, it follows that
$-\overline{\alpha}_x(\fra)$ is the greatest root of  $b_{g,(x,\lambda)}(s)/(s+1)$. By replacing $X$ 
with an open neighborhood of $x$ we may assume that $b_{\fra, x}(s) = b_{\fra}(s)$. On the other hand, it is shown in \cite[Theorem~1.1]{Mustata} that 
$$b_{\fra}(s) = {b}_{g}(s)/(s+1).$$
Since ${b}_{g, (x,\lambda)}(s)$ divides ${b}_{g}(s)$ (see e.g. the discussion at the beginning of \cite[\S6]{MP3}), we obtain the desired result.
\end{proof}

We recall that in the case of hypersurfaces, there exists also a close relationship between minimal exponents and the $V$-filtration (see e.g. \cite{Saito-MLCT}, and also \cite{MP3}). On the other hand, for subschemes of higher codimension, 
as in Remark \ref{local-coh-comment} a connection with (the several functions version of) the $V$-filtration seems to be more suitable in the alternative context of the Hodge filtration on local cohomology.



\end{document}